\documentclass[11pt]{article}
\usepackage[utf8]{inputenc}

\usepackage{tkz-graph}

\usepackage{cancel}
\usepackage[normalem]{ulem}
\usepackage[T1]{fontenc} 
\usepackage{color}
\usepackage[leqno]{amsmath}
\usepackage{amsfonts, amsthm, amssymb, commath} 
\usepackage{mathtools}
\usepackage{graphicx}
\usepackage{mathrsfs}
\usepackage{caption}
\usepackage{subcaption}
\usepackage{listings}
\usepackage[]{eufrak}
\usepackage[]{hyperref}
\usepackage[]{float}
\usepackage{tikz}
\usetikzlibrary{calc, arrows,matrix,positioning,scopes, intersections}
\usetikzlibrary{shapes.geometric}
\usetikzlibrary{decorations.markings}
\usepackage{tikz-cd}
\usetikzlibrary{cd}
\usepackage{booktabs}
\usepackage{enumerate}
\usepackage{eufrak}
\usepackage{MnSymbol}
\usepackage{tikz}
\usepackage{algorithm,algpseudocode}
\usetikzlibrary{matrix,decorations.pathreplacing,calc}
\pgfkeys{tikz/mymatrixenv/.style={decoration=brace,every left delimiter/.style={xshift=3pt},every right delimiter/.style={xshift=-3pt}}}
\pgfkeys{tikz/mymatrix/.style={matrix of math nodes,left delimiter={(}, right delimiter={)}, inner sep=1pt,column sep=1.5em,row sep=0.5em,nodes={inner sep=0pt}}}
\pgfkeys{tikz/mymatrixbrace/.style={decorate,thick}}

\usepackage{cuted}
\newcommand{\ca}[1]{\mathcal{#1}}

\newcommand{\bb}[1]{\mathbb{#1}}

\newcommand{\argmin}{\mathop{\rm arg\,min}}

\makeatletter
\newcommand{\leqnomode}{\tagsleft@true}
\newcommand{\reqnomode}{\tagsleft@false}
\makeatother
\renewcommand{\theenumi}{\alph{enumi}}

\newtheorem{theorem}{Theorem}[section]
\newtheorem{corollary}{Corollary}[theorem]
\newtheorem{lemma}[theorem]{Lemma}
\newtheorem{proposition}[theorem]{Proposition}
\newtheorem{definition}[theorem]{Definition}

\newtheorem{remark}[theorem]{Remark}

\tikzset{%
    add/.style args={#1 and #2}{
        to path={%
 ($(\tikztostart)!-#1!(\tikztotarget)$)--($(\tikztotarget)!-#2!(\tikztostart)$)%
  \tikztonodes},add/.default={.2 and .2}}
}  
\usepackage{cite}
\usepackage{balance}
\usepackage[margin=1.0in]{geometry}
\usepackage[toc,page]{appendix}
\begin{document}
\title{A Note on Nesterov's Accelerated Method in Nonconvex Optimization: a Weak Estimate Sequence Approach}
\author{Jingjing Bu and Mehran Mesbahi\thanks{The authors are with the University of Washington, Seattle; Emails:{\tt \{bu+mesbahi\}@uw.edu}}}
\date{}
\maketitle
\begin{abstract}
  We present a variant of accelerated gradient descent algorithms, adapted from Nesterov's optimal first-order methods, for weakly-quasi-convex and weakly-quasi-strongly-convex functions. We show that by tweaking the so-called estimate sequence method, the derived algorithm achieves optimal convergence rate for weakly-quasi-convex and weakly-quasi-strongly-convex in terms of oracle complexity. In particular, for a weakly-quasi-convex function with Lipschitz continuous gradient, we require $O(\frac{1}{\sqrt{\varepsilon}})$ iterations to acquire an $\varepsilon$-solution; for weakly-quasi-strongly-convex functions, the iteration complexity is $O\left( \ln\left(\frac{1}{\varepsilon}\right) \right)$. Furthermore, we discuss the implications of these algorithms for linear quadratic optimal control problem.
\end{abstract}
\section{Introduction}
\label{sec:intro}
Nesterov's seminal work~\cite{nesterov1983method} on \emph{Accelerated Gradient Descent}
has had profound implications on the design of first-order methods for large-scale convex optimization.
%
In recent years, applications in machine learning and statistics have lead
to an increased emphasis on large-scale \emph{nonconvex} problems. In the nonconvex realm, it is difficult to ensure global convergence of accelerated methods. Nesterov {\em et al.} in a recent work~\cite{nesterov2019primal} proposes a primal-dual acceleration scheme that works for both convex and nonconvex objectives. As a particular case, it was proven that the new scheme proposed~\cite{nesterov2019primal} has almost optimal convergence rate when applied to $\gamma$-weakly-quasi-convex functions. The notion of \emph{weakly-quasi-convex} is introduced by Hardt {\em et al.} in~\cite{hardt2016gradient} in order to guarantee the convergence of stochastic gradient descent in learning dynamical systems. Guminov {\em et al.} in~\cite{guminov2017accelerated} further analyzed convergence rate of first-order algorithm, i.e., gradient descent and subspace optimization method~\cite{narkiss2005sequential} for this class of functions. Right after the completion of this work, we became
aware of~\cite{hinder2019} on optimal methods for ``quasar-convex'' functions that are also the focus of this work.
Our approach to devising acceleration
for this class of functions is, in the meantime, 
distinct from that adopted in~\cite{hinder2019} and relies on the estimate sequence machinery~\cite{nesterov2013introductory}.
 Our work is partially motivated by a class of problems in control and machine learning in the context of data-driven decision-making. A typical scenario is to use observed data to directly synthesize controllers for a dynamical systems. Linear-Quadratic-Regulator (LQR) is one popular paradigm for such control synthesis. As subsequently pointed out in this paper,  the LQR cost as a function of the feedback gain is indeed $\gamma$-weakly-quasi-convex; see \S\ref{LQR}. \par
In this manuscript, we focus on designing accelerated gradient methods~\cite{nesterov2013introductory} for weakly-quasi-convex functions. Indeed, we introduce a notion of $(\gamma, \mu)$-weakly-quasi-strongly-convex functions, which subsumes the function class of weakly-quasi-convex functions;\footnote{When $\mu=0$, then this function class is exactly the class of \emph{weakly-quasi-convex} functions.} we then proceed to design a generative accelerated gradient descent framework for this function class. Furthermore, we show that the proposed algorithms achieve optimal oracle complexity for weakly-quasi-convex and weakly-quasi-strongly-convex functions.
 \section{Preliminaries}
 \label{sec:prelim}
 Consider the unconstrained minimization problem,
 \begin{align*}
   \min_{x \in \bb R^n} f(x),
   \end{align*}
   where $f: \bb R^n \to \bb R$ is $L$-smooth (not necessarily convex) and \emph{bounded below}. We shall be mainly concerned with class of weakly-quasi-convex functions and its stronger variants.\footnote{By ``stronger'' we mean functions that enjoy more regularity in addition to being weakly-quasi-convex.} We note that these functions have a global minimum. Denote the (nonempty) set of minimizers by $\chi = \argmin_{x \in \bb R^n} f(x)$.\par
   We next recall several definitions.
 \begin{definition}[\cite{hardt2016gradient}]
  The function $f$ is $\gamma$-weakly-quasi-convex if
\begin{align*}
  \gamma [f(x) - f(x^*)] \le \langle \nabla f(x), x-x^* \rangle
\end{align*}
for some $\gamma \in (0, 1]$. 
\end{definition}
Note that if $\gamma = 1$, the above definition coincides with \emph{weak convexity} for $C^1$ functions.
\begin{definition}
  The function $f$ satisfies the quadratic growth condition if
  \begin{align*}
    f(x) - f(x^*) \ge \frac{1}{2} \zeta  \langle \nabla f(x), \nabla f(x)\rangle,
    \end{align*}
    where $\zeta >0$ is some fixed constant.
\end{definition}
Generally, in order to ensure \emph{linear convergence} of a ``vanilla'' gradient descent algorithm, stronger assumptions on the function class are needed. Next, we introduce two assumptions that we shall subsequently see have implications for the linear convergence rate of gradient descent. 
\begin{definition}
 The function $f$ is $(\gamma, \mu)$-weakly-quasi-strongly-convex if
\begin{align*}
  f(x) -f(x^*) \le \frac{1}{\gamma} \langle \nabla f(x), x-x^*\rangle - \frac{\mu}{2} \|x-x^*\|^2.
\end{align*}
\end{definition}
This definition can be considered as a relaxed weak-strong-convexity in the nonconvex setting. In a sense, it resembles strong convexity: rewriting the above inequality yields,
\begin{align*}
  f(x^*) \ge f(x) + \frac{1}{\gamma} \langle \nabla f(x), x^* - x \rangle + \frac{\mu}{2} \|x-x^*\|^2.
\end{align*}
We shall denote the class of $L$-smooth and $(\gamma, \mu)$-weakly-quasi-strongly-convex functions by $\ca W_{L, \gamma, \mu}$. \par
Next, we consider the function class $\ca{WQ}_{L,\gamma, \mu}$ consisting of $\gamma$-weakly-quasi-convex functions satisfying the quadratic growth condition with constant $\mu$.
\begin{definition}
  The function $f \in \ca {WQ}_{\gamma, \mu}(\bb R^n)$ if
  \begin{align*}
    \frac{\mu}{2}\|x-x^*\|^2 \le f(x)-f(x^*) \le \frac{1}{\gamma} \langle \nabla f(x), x-x^*\rangle.
    \end{align*}
  \end{definition}
  Finally, we recall the definition of gradient dominated functions.
\begin{definition}[\cite{polyak1963gradient}]
  The function $f$ is gradient dominated with constant $\tau$ if
  \begin{align*}
    \tau \left[f(x) - f(x^*) \right]\le \frac{1}{2} \langle \nabla f(x), \nabla f(x)\rangle.
    \end{align*}
\end{definition}
\section{Relationships between function classes}
For various notions of regularity and how they are related, we refer to the work by Karimi {\em et al.}~\cite{karimi2016linear}. In this section, we examine the function classes $\ca{W}_{L, \gamma, \mu}$ and $\ca{WQ}_{L, \gamma, \mu}$. We first note that $\ca{WQ}_{L, \gamma, \mu} \subseteq \ca{W}_{L, \frac{1}{1/\gamma + a/(\mu \gamma)}, a}$ for any $a > 0$.
\begin{proposition}
  \label{prop:wq2w}
  For every $a > 0$, if $f \in \ca{WQ}_{L, \gamma, \mu}$, then $f \in \ca{W}_{L, \frac{1}{1/\gamma +a/(\mu \gamma)}, a}$, i.e.,
  \begin{align*}
    f(x) - f(x^*) \le \left( \frac{1}{\gamma} + \frac{a}{\mu \gamma} \right) \langle \nabla f(x), x-x^* \rangle - \frac{a}{2}\|x-x^*\|^2.
    \end{align*}
  \end{proposition}
  \begin{proof}
    It suffices to observe that if $f \in \ca{WQ}_{L, \gamma, \mu}$, then
    \begin{align*}
      \frac{1}{2}\|x-x^*\|^2 \le \frac{1}{\gamma \mu} \langle \nabla f(x), x-x^*\rangle.
      \end{align*}
     As such,
      \begin{align*}
        f(x) -f(x^*) &\le \frac{1}{\gamma}\langle \nabla f(x), x-x^*  \rangle + \frac{a}{2}\|x-x^*\|^2 - \frac{a}{2}\|x-x^*\|^2 \\
                     &\le \left( \frac{1}{\gamma} + \frac{a}{\gamma \mu}\right) \langle \nabla f(x), x-x^* \rangle - \frac{a}{2}\|x-x^*\|^2.
        \end{align*}
    \end{proof}
We next observe that if $f \in \ca W_{L, \gamma, \mu}$, then $f$ is gradient dominated.
\begin{lemma}
    \label{lemma:wqsc_gdom}
    If $f \in \ca W_{L, \gamma, \mu}$, then
    \begin{align*}
      \mu \gamma^2 [f(x) - f(x^*)] \le \frac{1}{2}\langle \nabla f(x), \nabla f(x)\rangle.
    \end{align*}
    \end{lemma}
    \begin{proof}
      If $f \in \ca W_{L, \gamma, \mu}$, then
      \begin{align*}
        f(x) -f(x^*) \le \frac{1}{\gamma} \langle \nabla f(x), x-x^*\rangle - \frac{\mu}{2} \|x-x^*\|^2.
    \end{align*}
    Noting that for every $\rho > 0$,
\begin{align*}
  \langle \sqrt{\frac{\rho}{2}} \nabla f(x) - \sqrt{\frac{1}{2\rho}} (x-x^*), \sqrt{\frac{\rho}{2}} \nabla f(x) - \sqrt{\frac{1}{2\rho}} (x-x^*) \rangle \ge 0,
  \end{align*}
  it follows that,
    \begin{align*}
      \langle \nabla f(x), x-x^*\rangle \le \frac{\rho}{2}\langle  \nabla f(x), \nabla f(x)\rangle + \frac{1}{2\rho} \langle x-x^*, x-x^*\rangle.
      \end{align*}
      Choosing $\rho$ such that $\frac{1}{2\rho \gamma}=\frac{\mu}{2}$, i.e., $\rho = \frac{1}{\gamma \mu}$, we have
      \begin{align*}
        f(x) - f(x^*) \le \frac{1}{2\mu \gamma^2} \langle \nabla f(x), \nabla f(x)\rangle.
        \end{align*}
    \end{proof}
    The following result is now immediate.
    \begin{corollary}
      If $f \in \ca{W}_{L, \gamma, \mu}$, then $f \in \ca{WQ}_{L, \gamma, 4\mu\gamma^2}$.
      \end{corollary}
      \begin{proof}
        It suffices to observe that by Theorem $1$ in~\cite{karimi2016linear}, $f$ satisfies \emph{quadratic growth} condition,
        \begin{align*}
          f(x) - f(x^*) \ge 2 \mu \gamma^2 \|x-x^*\|^2.
          \end{align*}
        \end{proof}
\section{Convergence rate of gradient descent for function classes $\ca W_{L, \gamma, \mu}$ and $\ca {WQ}_{L, \gamma, \mu}$}
In this section, we examine the convergence rates of gradient descent for $W_{L, \gamma, \mu}$ and $\ca{WQ}_{L, \gamma, \mu}$, i.e., 
\begin{align}
  \label{eq:gd}
  x_{k+1} = x_k - h_k \nabla f(x_k),
  \end{align}
  where $h_k$ is the stepsize.
The case where $\mu = 0$ has been examined in~\cite{guminov2017accelerated}. 
Note that in this case, $\ca{W}_{L, \gamma, 0} \equiv \ca{WQ}_{L,\gamma, 0}$.
\begin{theorem}[Theorem $1$ in~\cite{guminov2017accelerated}]
  If $f \in \ca W_{L, \gamma, 0}$, then the sequence $\{x_k\}$ generated by gradient descent with stepsize $1/L$, i.e.,
  \begin{align*}
    x_{k}  = x_{k-1} - \frac{1}{L} \nabla f(x_{k-1}),
    \end{align*}
    satisfies
    \begin{align*}
      f(x_k) - f^* \le \frac{L \|x_0 - x^*\|^2}{ \gamma (k+1)}.
      \end{align*}
\end{theorem}
We shall next establish that when $\mu > 0$, the convergence rate of gradient descent~\eqref{eq:gd} is linear.
\begin{lemma}
  \label{lemma:convergence_wqsc}
  If $f \in \ca W_{L, \gamma, \mu}$, then the sequence $\{x_k\}$ generated by gradient descent with stepsize $\gamma/L$, i.e.,
\begin{align*}
  x_{k}= x_{k-1} - \frac{\gamma}{L} \nabla f(x_{k-1}),
  \end{align*}
  satisfies
  \begin{align*}
    \|x_{k+1} - x^*\|^2 \le \left( 1 - \frac{\gamma ^2 \mu }{L} \right)^{k+1} \|x_0 - x^*\|^2.
    \end{align*}
  \end{lemma}
  \begin{proof}
    Putting $r_k \coloneqq \|x_k - x^*\|$, we have
    \begin{align}
      \label{lemma2_myeq1}
      \begin{split}
      r_{k+1}^2 &= \norm{ x_{k} - \frac{\gamma}{L} \nabla f(x_k) - x^*}^2 \\
                &=r_k^2 - \frac{2 \gamma }{L} \langle \nabla f(x_k), x_k - x^* \rangle + \frac{\gamma^2 }{L^2} \|\nabla f(x_k)\|^2.
                \end{split}
      \end{align}
      We note that,
      \begin{align}
        \begin{split}
        \frac{1}{\gamma} \langle \nabla f(x_k), x^* - x_k \rangle &\le  f^* - f(x_k) - \frac{\mu}{2} \|x_k - x^*\|^2 \\
                                                                  &\le  -\frac{1}{2L} \|\nabla f(x_k)\|^2 - \frac{\mu}{2} \|x_k - x^* \|^2.
                                                                    \end{split}\label{lemma2_myeq2}
        \end{align}
        Substituting~\eqref{lemma2_myeq2} into~\eqref{lemma2_myeq1} completes the proof.
    \end{proof}
We now show that the proposed algorithms achieve optimal convergence rate for weakly-quasi-convex, weakly-quasi-strongly-convex, and gradient dominated functions in terms of the respective oracle models. 
    \begin{remark}
      As we have established in Lemma~\ref{lemma:wqsc_gdom} that if $f \in \ca W_{L, \gamma, \mu}$, then $f \in \ca{GD}_{L, {\mu}\gamma^2}$. It is possible to follow a similar procedure in~\cite{polyak1963gradient} Theorem $4$ to arrive at the convergence rate,
      \begin{align*}
            f(x_{k+1}) - f(x^*) \le \left( 1 - \frac{\mu \gamma^2 }{L }\right)^{k+1} (f(x_0)-f(x^*)),
        \end{align*}
        with constant stepsize $1/L$. It appears that the direct argument in Lemma~\ref{lemma:convergence_wqsc} proves a somewhat stronger convergence result; namely, the iterates converge at the same linear rate.
      \end{remark}
        By Proposition~\ref{prop:wq2w}, it is straightforward to conclude that the convergence rate for the function class $\ca{WQ}_{L, \gamma, \mu}$ is given by,
        \begin{align*}
          \|x_{k+1} - x^*\|^2 \le \left( 1 - \frac{\gamma^2 \mu}{L} \frac{a/\mu}{(1+a/\mu)^2}\right)^{k+1} \|x_0 -x^*\|^2.
        \end{align*}
                        Maximizing the quantity
                        $$\frac{a/\mu}{(1+a/\mu)^2}$$
                        over $a \in (0, \infty)$, we obtain $a = \mu$. The convergence rate gradient descent on $\ca {WQ}_{L, \gamma, \mu}(\bb R^n)$ is now summarized as follows.
      \begin{lemma}
        \label{lemma:wq_gd_rate}
                          If $f \in \ca {WQ}_{L, \gamma, \mu}$, then the sequence $\{x_k\}$ generated by gradient descent with stepsize $\gamma/2L$, i.e.,
\begin{align*}
  x_{k+1}= x_{k} - \frac{\gamma}{2L} \nabla f(x_{k}),
  \end{align*}
  satisfies
  \begin{align*}
    \|x_{k+1} - x^*\|^2 \le \left( 1 - \frac{\gamma ^2 \mu }{4L} \right)^{k+1} \|x_0 - x^*\|^2.
    \end{align*}
  \end{lemma}
    \section{Accelerated gradient descent for function class $\ca W_{L, \gamma, \mu}$}
In this section, we shall develop the accelerated gradient method over the function class $\ca W_{L, \gamma, \mu}$. Our approach closely follows Nesterov's original treatment via an estimate sequence~\cite{nesterov2013introductory}. However as our function is nonconvex, a global estimate sequence is almost impossible to construct. It turns out that by slightly relaxing the definition of estimate sequence, one can construct \emph{weakly estimation sequence} achieving the same purpose as in~\cite{nesterov2013introductory}. 
\subsection{Weak Estimate Sequence}
In this section, we modify the definition of \emph{estimate sequence} by Nesterov and introduce the notion of a \emph{weak estimate sequence}. As we shall see, a weak estimate sequence can accommodate acceleration for nonconvex problems. Note that our definition is a close variant of Nesterov's original definition and all subsequent treatment follows closely Section $2.2$ in~\cite{nesterov2013introductory}.
\begin{definition}
  A weak estimate sequence is a sequence of functions $\{\phi_k(x)\}_{k=0}^{\infty}$ and a sequence of positive scalars $\{\lambda_k\}_k^{\infty}$ such that
\begin{align*}
  \lim_{k \to \infty} \lambda_k = 0 \quad \text{and } \quad \phi_k(x^*) \le (1-\lambda_k) f(x^*) + \lambda_k \phi_0(x^*),
\end{align*}
where $x^* \in \argmin f(x)$.
\end{definition}
We next present a proposition showing that we can utilize a weak estimate sequence to solve optimization problem analogous to Lemma $2.2.1$ in~\cite{nesterov2013introductory}.
\begin{proposition}
  If for some sequence $\{x_k\}_{k=0}^{\infty}$ we have
\begin{align*}
  f(x_k) \le \phi_k^* \equiv \min_{x \in \bb R^n} \phi_k(x),
\end{align*}
then $f(x_k) - f^* \le \lambda_k \left( \phi_0(x^*) - f^* \right)$.
\end{proposition}
\begin{proof}
  It suffices to observe that,
\begin{align*}
  f(x_k) \le \phi_k^* = \min_{x \in \bb R^n} \phi_k(x) \le \phi_k(x^*) \le (1-\lambda_k) f(x^*) + \lambda_k \phi_0(x^*),
\end{align*}
where the last inequality follows from the definition of a weak estimate sequence.
\end{proof}
We next present a proposition describing how we can construct a weak estimate sequence;
this observation is a slight variant of Lemma $2.2.2$~\cite{nesterov2013introductory}.
\begin{proposition}
  Suppose that
\begin{itemize}
  \item
    $f$ is $(\gamma, \mu)$-weakly-quasi-strongly-convex,
  \item $\phi_0$ is an arbitrary function on $\bb R^n$,
  \item $\{y_k\}$ is an arbitrary sequence in $\bb R^n$,
    \item $\{\alpha_k\}_{k=0}^{\infty}$: $\alpha_k \in (0,1)$, $\sum_{k=0}^{\infty} \alpha_k = \infty$, and 
      \item $\lambda_0 = 1$.
      \end{itemize}
      Then the sequences $\{\phi_k(x)\}_{k=0}^{\infty}$ and $\{\lambda_k\}_{k=0}^{\infty}$ recurcisely defined by:
      \begin{align}
        \lambda_{k+1} &= (1-\alpha_k) \lambda_k, \nonumber \\
        \phi_{k+1} (x) &= (1-\alpha_k) \phi_k(x) + \alpha_k \left[ f(y_k) + \frac{1}{\gamma} \langle \nabla f(y_k),x-y_k\rangle + \frac{\mu}{2}  \|x-y_k\|^2\right], \label{eq:recursive_lower_estimate}
        \end{align}
        is a weak estimate sequence.
\end{proposition}
\begin{proof}
  Putting
  \begin{align*}
    q_k(x; y_k) = f(y_k)  + \frac{1}{\gamma} \langle \nabla f(y_k),x-y_k\rangle + \frac{\mu}{2} \|x-y_k\|^2,
    \end{align*}
    we note that $q_k(x^*; y_k) \le f(x^*)$ for all $k \ge 1$. Now observe that $\phi_0(x^*) = (1-\lambda_0) f(x^*) + \lambda_0 \phi_0(x^*)$. Furthermore, suppose $\{\phi_j\}_{j=0}^k$ is a weak estimate sequence up to $k$. Then
    \begin{align*}
      \phi_{k+1}(x^*) &\le (1-\alpha_k) \phi_k(x^*) + \alpha_k q_k(x^*; y_k) \\
                      &\le (1-\alpha_k) \left[(1-\lambda_k) f(x^*) + \lambda_k \phi_0(x^*)\right] + \alpha_k f(x^*) \\
                      &\le (1-\alpha_k) \lambda_k \phi_0(x^*) + (1-\lambda_{k+1}) f(x^*) \\
                      &= \lambda_{k+1} \phi_0(x^*) + (1-\lambda_{k+1}) f(x^*).
      \end{align*}
  \end{proof}
  \subsection{An algorithm and its convergence}
 In order to devise the acceleration scheme, one
 needs to choose the sequence $\{y_k\}_{k=0}^{\infty}$ for which successive \emph{weak lower estimate} of $f^*$ can be formed and a function $\phi_0$ that is easy to minimize. One choice of $\phi_0$ is a simple quadratic function suggested by Nesterov~\cite{nesterov2013introductory}. The next proposition examines how $\phi_k^*\equiv \min_{x \in \bb R^n} \phi_k(x)$ varies under the construction of weak estimate sequence; we shall see that $\phi^*$  directly suggests choosing the sequence $\{y_k\}_{k=0}^{\infty}$ and $\{x_k\}_{k=0}^{\infty}$. Throughout the presentation we will use the sequence $\{v_k\}_{k=0}^{\infty}$ to denote the unique minimizers of $\{\phi_k\}_{k=0}^{\infty}$.\footnote{As $\phi_k(x)$'s are strongly convex functions, the minimizer is unique.} Needless to say, this lemma is a variant of Lemma $2.2.3$ in~\cite{nesterov2013introductory} and its proof has been adapted accordingly.
  \begin{lemma}
    \label{lemma:phi_star}
    Let $\phi_0(x) = \phi_0^* + \frac{\gamma_0}{2} \|x-v_0\|^2$, where $\phi_0^*$ is a scalar. Then the process~\eqref{eq:recursive_lower_estimate} preserves the canonical form of functions,
    \begin{align}
      \label{eq:canonical_phi}
      \phi_k(x) \equiv \phi_k^* + \frac{\gamma_k}{2} \|x-v_k\|^2,
      \end{align}
      where the sequences $\{\gamma_k\}$, $\{v_k\}$ and $\{\phi_k^*\}$ are given as,
        \begin{align}
          \nonumber
          \gamma_{k+1} &= (1-\alpha_k) \gamma_k +  \frac{ \alpha_k \mu}{\gamma}, \\
          \nonumber v_{k+1} &= \frac{1}{\gamma_{k+1}} \left[(1-\alpha_k) \gamma_k v_k + \frac{\alpha_k \mu}{\gamma} y_k - \frac{\alpha_k}{ \gamma} \nabla f(y_k)\right],\\
          \label{eq:phi_star} \phi_{k+1}^* &= (1-\alpha_k) \phi_k^* + \alpha_k f(y_k) - \frac{\alpha_k^2 }{2 \gamma^2 \gamma_{k+1}} \|\nabla f(y_k)\|^2 + \frac{\alpha_k(1-\alpha_k)\gamma_k}{\gamma_{k+1}} \left[ \frac{\mu}{2 \gamma} \|y_k - v_k\|^2 + \frac{1}{\gamma} \langle \nabla f(y_k), v_k-y_k\rangle\right].
        \end{align}
    \end{lemma}
    \begin{proof}
      We first observe that $\phi_0''(x) = \gamma_0 I_n$. Suppose the statement of the lemma holds for $j \le k$; then
      \begin{align*}
        \phi_{k+1}''(x) = (1-\alpha_k) \phi_k''(x) + \alpha_k \mu I = \left[(1-\alpha_k)\gamma_k + \alpha_k \mu\right]I_n \equiv \gamma_{k+1}I_n,
        \end{align*}
        where $I_n$ is the $n \times n$ real identity matrix. This shows that the canonical form holds for~\eqref{eq:canonical_phi}.\\
        Moreover noting that,
        \begin{align}
          \label{eq:phi_inter}
          \phi_{k+1}(x) = (1-\alpha_k) \left( \phi_k^* + \frac{\gamma_k}{2} \|x-v_k\|^2\right) + \alpha_k \left[f(y_k) + \frac{1}{\gamma}\langle \nabla f(y_k), x-y_k\rangle + \frac{\mu}{2}\|x-y_k\|^2\right],
          \end{align}
          the first-order optimality condition implies $\nabla \phi_{k+1}(v_{k+1}) = 0$\footnote{Recall $v_{k+1}$ is the uniqe minimizer for $\phi_{k+1}(x)$.}, i.e.,
          \begin{align*}
            (1-\alpha_k)\gamma_k(v_{k+1}-v_k) + \frac{\alpha_k}{ \gamma} \nabla f(y_k) + \alpha_k \mu (v_{k+1}-y_k) =0,
            \end{align*}
           that can be re-written as,
          \begin{align*}
           v_{k+1} - y_k =  \frac{1}{\gamma_{k+1}} \left[(1-\alpha_k)\gamma_k(v_k-y_k) - \frac{\alpha_k}{ \gamma} \nabla f(y_k)\right].
            \end{align*}
            It thus follows that,
            \begin{align*}
              \frac{\gamma_{k+1}}{2} \|v_{k+1} - y_k\|^2 = \frac{1}{2 \gamma_{k+1}} \left[ (1-\alpha_k)^2 \gamma_k^2 \|v_k-y_k\|^2 - \frac{2 \alpha_k (1-\alpha_k)\gamma_k}{ \gamma} \langle \nabla f(y_k), v_k-y_k\rangle + \frac{\alpha_k^2}{ \gamma^2} \|\nabla f(y_k)\|^2 \right].
              \end{align*}
              Now in view of the recursion of $\{\phi_k(x)\}$, we have
              \begin{align}
                \label{eq:phi_inter2}
                \begin{split}
                \phi_k^* + \gamma_k\|y_k - v_{k+1}\|^2 = \phi_k(y_k)  & \\
& = (1-\alpha_k)\phi_k^* + \frac{(1-\alpha_k)\gamma_k}{2}\|y_k-v_k\|^2 + \alpha_k f(y_k).
\end{split}
                \end{align}
              Substituting this relation into~\eqref{eq:phi_inter2} and observing that,
              \begin{align*}
                (1-\alpha_k)\frac{\gamma_k}{2} - \frac{1}{2 \gamma_{k+1}} (1-\alpha_k)^2 \gamma_k^2 = (1-\alpha_k) \frac{\gamma_k}{2}\left(1 - \frac{(1-\alpha_k)\gamma_k}{\gamma_{k+1}}\right) = (1-\alpha_k)\frac{\gamma_k}{2} \cdot \frac{\alpha_k \mu}{\gamma_{k+1}},
                \end{align*}
                we can conclude~\eqref{eq:phi_star}.
      \end{proof}
      Now let us describe how we can design the accelerated gradient descent algorithm. Suppose we have already chosen $x_k$ satisfying $f(x_k) \le \phi_k^*$. Then in view of Lemma~\ref{lemma:phi_star}, we have
      \begin{align}
        \label{eq:phi_star2}
        \phi_{k+1}^* \ge (1-\alpha_k) f(x_k) + \alpha_k f(y_k) - \frac{\alpha_k^2 }{2 \gamma^2 \gamma_{k+1}}\|\nabla f(y_k)\|^2 + \frac{\alpha_k(1-\alpha_k)\gamma_k}{\gamma_{k+1} \gamma }   \langle \nabla f(y_k), v_k-y_k\rangle.
      \end{align}
      If $f$ had been a convex function, we could underestimate $f(x_k) \ge f(y_k) + \langle \nabla f(y_k), x_k-y_k\rangle$. Then setting $y_k$ to be a linear combination of $v_k$ and $x_k$ in such a way that $\langle \nabla f(y_k), \cdot \rangle$ would vanish, a simple gradient step would allow us to choose $x_{k+1}$ (see details in Section $2.2$~\cite{nesterov2013introductory}). However since we are concerned with a nonconvex function $f$, this strategy does not work. Let us examine this situation more closely:
\par
\vspace{3pt}
\emph{We want to choose $y_k$ and $x_{k+1}$ in such a way that both can easily be computed and  satisfy $f(x_{k+1}) \le \phi_{k+1}^*$. 
A promising strategy is to reduce ~\eqref{eq:phi_star2} to,
\begin{align*}
        \phi_{k+1}^* \ge  f(y_k) - \frac{\alpha_k^2 }{2\gamma^2 \gamma_{k+1}}\|\nabla f(y_k)\|^2.
  \end{align*}
  This would then enable us to choose $x_{k+1}$ by a simple gradient step. This reduction however would require that,
\begin{align}
  \label{eq:line_search_condition}
  f(x_{k}) + \frac{\alpha_k \gamma_k}{\gamma_{k+1} \gamma} \langle \nabla f(y_k), v_k - y_k\rangle + \frac{\alpha_k \gamma_k \mu}{2\gamma_{k+1}} \|y_k - v_k\|^2 \ge f(y_k).
\end{align}}
\vspace{3pt}
      \par As suggested in~\cite{nesterov2019primal}, a line search over the line segment $[x_k, v_k]$ would guarantee the above inequality~\eqref{eq:line_search_condition} hold. The next result and its proof are extracted from proof of Lemma $1$ in~\cite{nesterov2019primal}.
\begin{proposition}
  \label{prop:line_search}
  If we set $y_k$ as
      \begin{align*}
        y_k = v_k + \beta_k(x_k - v_k), \text{ where }
        \beta_k = \argmin_{\beta \in [0, 1]} f( v_k + \beta(x_k-v_k)),
        \end{align*}
        then we have
        \begin{align*}
          f(y_k) \le f(x_k), \text{ and } \langle \nabla f(y_k), v_k -y_k\rangle \ge 0.
          \end{align*}
          \end{proposition}
                                                                                                                                                                                                                                                                                                                                                                                   \begin{proof}
        First, it is clear that $f(y_k) \le f(x_k)$. Next we examine three cases:
        \begin{enumerate}
          \item If $\beta_k = 0$, then $y_k = v_k$.
            \item If $\beta_k \in (0,1)$, by optimality condition $\langle \nabla f(y_k), x_k-v_k \rangle = 0$. As $y_k = v_k + \beta_k(x_k - v_k)$, we have $\langle \nabla f(y_k), y_k-v_k \rangle = 0 $.
              \item If $\beta_k = 1$, $y_k = x_k$ and $\langle \nabla f(y_k), x_k - v_k \rangle \le 0$.
          \end{enumerate}
          \end{proof}                                                                                                                                                                           Note that Proposition~\ref{prop:line_search} implies the relation~\eqref{eq:line_search_condition}. 
Now suppose that $y_k$ is chosen such that inequality~\eqref{eq:line_search_condition} holds;
it then follows that,
      \begin{align*}
        \phi_{k+1}^* \ge  f(y_k) - \frac{\alpha_k^2 }{2 \gamma^2 \gamma_{k+1}}\|\nabla f(y_k)\|^2.
      \end{align*}
      Then setting $x_{k+1} = y_k - \frac{1}{L}\nabla f(y_k)$ would guarantee the sufficient decrease of the function value. This essentially requires
      \begin{align*}
\frac{\alpha_k^2 }{2 \gamma^2 \gamma_{k+1}} = \frac{1}{2L},
        \end{align*}
        that is, $\alpha_k$ solves
        \begin{align*}
          \frac{L \alpha_k^2}{ \gamma^2} = (1-\alpha_k) \gamma_k + \alpha_k \mu.
          \end{align*}
          We are now in the position to describe our algorithm:
          \begin{algorithm}[H]
            \caption{Accelerated gradient descent for quasi-strongly-convex and quasi-convex functions}
            \label{alg1}
            \begin{algorithmic}[1]
              \State \text{Initialize } $x_0 \in \bb R^n$ and $\gamma_0 > 0$. \text{Set } $v_0 = x_0$.
              \If{ $k \ge 0$}
              \State \text{Compute $\alpha_k \in (0, 1)$ satisfying}
              \begin{align*}
          \frac{L \alpha_k^2}{ \gamma^2} = (1-\alpha_k) \gamma_k +  \alpha_k \mu.
                \end{align*}
                \State Set: $\gamma_{k+1} = (1-\alpha_k)\gamma_k+\alpha_k \mu$.
                \State
                Choose $y_k = v_k + \beta_k(x_k- v_k)$ with $\beta_k \in [0, 1]$ such that
                \begin{align*}
  f(x_{k}) + \frac{\alpha_k \gamma_k}{\gamma_{k+1} \gamma} \langle \nabla f(y_k), v_k - y_k\rangle + \frac{\alpha_k \gamma_k \mu}{2\gamma_{k+1}} \|y_k - v_k\|^2 \ge f(y_k).
                  \end{align*}
                  \State Set: $x_{k+1} = y_k - \frac{1}{L} \nabla f(y_k)$.
                  \State Set:
                  \begin{align*}
                    v_{k+1} = \frac{1}{\gamma_{k+1}} \left[(1-\alpha_k) \gamma_k v_k + \alpha_k \mu y_k - \frac{\alpha_k}{ \gamma} \nabla f(y_k)\right].
                    \end{align*}
                \EndIf
              \end{algorithmic}
            \end{algorithm}                                                                                            Note that $\beta_k$ Step $5$ always exists by Proposition~\ref{prop:line_search}. A backtracking line search thereby can be adopted.
            Next, we establish the convergence rate of the proposed algorithm.
            \begin{theorem}
              If the sequence $\{x_k\}_{k=0}^{\infty}$ is generated by Algorithm~\ref{alg1}, then
              \begin{align*}
                f(x_k)- f^* \le \lambda_k \left( f(x_0) - f^* + \frac{\gamma_0}{2}\|x_0-x^*\|^2\right),
                \end{align*}
                where $\lambda_0=1$ and $\lambda_k = \Pi_{i=0}^{k=1} (1-\alpha_i)$.
              \end{theorem}
              \begin{proof}
                If we choose $\phi_0(x) = f(x_0)+ \frac{\gamma_0}{2}\|x-v_0\|^2$. Then $f(x_0)=\phi_0^*$ and by contruction we have $f(x_k) \le \phi_k^*$.
                \end{proof}
                Next, let us estimate $\lambda_k$.
                \begin{lemma}
                  If we choose $\gamma_0 \ge \mu$, then
                  \begin{align*}
                    \lambda_k \le \min \left( \left( 1-\sqrt{\frac{\mu \gamma^2 }{L}}\right)^k, \frac{4L}{(2 \sqrt{L} + \gamma k \sqrt{\gamma_0})^2}\right).
                    \end{align*}
                  \end{lemma}
                  \begin{proof}
                    If $\gamma_k \ge \mu$, then
                    \begin{align*}
                      \gamma_{k+1} = \frac{L \alpha_k^2}{ \gamma^2} = (1-\alpha_k) \gamma_k + \alpha_k \mu\ge \mu.
                      \end{align*}
                      Since $\gamma_0 \ge \mu/\gamma$, by induction, it is true for all $k \ge 0$. So we have $\alpha_k \ge \sqrt{\frac{\mu \gamma}{L}}$.\par
                    Let us prove $\gamma_k \ge \gamma_0 \lambda_k$ by induction. It is evident that $\gamma_0 = \gamma_0 \lambda_0$. If this inequality holds for $j \le k$, we have
                    \begin{align*}
                      \gamma_{k+1} \ge (1-\alpha_k) \lambda_k \ge (1-\alpha_k) \gamma_0 \lambda_k = \gamma_0 \lambda_{k+1}.
                      \end{align*}
                      Hence $L\alpha_k^2 /\gamma^2 = \gamma_{k+1} \ge \gamma_0 \lambda_{k+1}$. \par
                      Putting $b_{k} = \frac{1}{\sqrt{\lambda_k}}$ and noting $\{\lambda_k\}$ is nonincreasing, we have
                      \begin{align*}
                        b_{k+1} - b_k &= \frac{\sqrt{\lambda_k} - \sqrt{\lambda_{k+1}}}{\sqrt{\lambda_k \lambda_{k+1}}} = \frac{\lambda_k - \lambda_{k+1}}{\sqrt{\lambda_k \lambda_{k+1}}(\sqrt{\lambda_k} + \sqrt{\lambda_{k+1}})} \ge \frac{\lambda_k - \lambda_{k+1}}{2\lambda_k\sqrt{\lambda_{k+1}}} \\
                        &= \frac{\lambda_k - (1-\alpha_k)\lambda_k}{2 \lambda_k \sqrt{\lambda_{k+1}}} = \frac{\alpha_k}{2 \sqrt{\lambda_{k+1}}} \ge \frac{\gamma}{2} \sqrt{\frac{\gamma_0}{L}}.
                        \end{align*}
                        Hence $b_k \ge 1 + \frac{ \gamma k}{2 }\sqrt{ \frac{\gamma_0}{L}}$ and the statement follows.
                    \end{proof}
                    We can now present the exact convergence rate of the proposed algorithm.
                    \begin{theorem}
                      \label{thrm:wqc_convergence_rate}
                      Putting $\gamma_0 = \max(L, \mu/\gamma)$, the sequence $\{x_k\}$ generated by Algorithm~(\ref{alg1}) satisfies
                      \begin{align*}
                        f(x_k) - f^* \le \min\left( \left( 1-\sqrt{\frac{\mu \gamma^2}{L}}\right)^k, \frac{4}{(2   +\gamma  k )^2} \right) L \|x_0-x^*\|^2.
                        \end{align*}
                      \end{theorem}
                      If $\mu =0$, namely for the $\gamma$-weakly-quasi-convex case, we obtain the following rate of convergence.
                      \begin{corollary}
                        If $f$ is $\gamma$-weakly-quasi-convex, then the sequence $\{x_k\}$ generated by Algorithm~(\ref{alg1}) satisfies
                        \begin{align*}
                        f(x_k) - f^* \le  \frac{4L\|x_0-x^*\|^2}{(2   +\gamma  k )^2} .
                          \end{align*}
                        \end{corollary}
                        \begin{remark}
                          This result seems to be better than the one provided by Theorem $4$ in~\cite{nesterov2019primal} for weakly-quasi-convex functions, where the rate is
                          \begin{align*}
                            f(x_k) - f(x^*) = O \left( \frac{L\|x_0-x^*\|^2}{\gamma^3 k^2}\right).
                            \end{align*}
                            We impove the dependence on $\gamma$ by a factor $1/\gamma$ (recall $\gamma < 1$). Moreover, the proof in~\cite{nesterov2019primal} relies on a restarting technique whereas the proof given here is directly adapted from Nesterov's estimate sequence method. \par
                            The convergence rate in Theorem~\ref{thrm:wqc_convergence_rate} is at the same order with the one given by Guminov et al.~\cite{guminov2017accelerated} via Sequential Subspace Optimization Method~\cite{narkiss2005sequential}. However, the algorithm 
                            in~\cite{guminov2017accelerated} requires solving a three-dimensional convex problem in each iteration. One the other hand, the algorithm proposed in this work merely requires a line search step.
                          \end{remark}
                      \section{Accelerated gradient descent for function class $\ca{WQ}_{L, \gamma, \mu}$}
                          For $f \in \ca{WQ}_{L, \gamma, \mu}$, we have established in Proposition~\ref{prop:wq2w}, that
                          \begin{align*}
                            f \in \ca{W}_{L, \frac{1}{1/\gamma +a/(\mu \gamma)}, a}
                           \end{align*}
                          for every $a > 0$. Note that by Lemma~\ref{lemma:wq_gd_rate}, the best choice of $a$ for the convergence rate is $a = \mu$. For completeness, we write down explicitly the algorithm for function class $\ca{WQ}_{L, \gamma, \mu}$ in Algorithm~\ref{alg2}.
          \begin{algorithm}[H]
            \caption{Accelerated Gradient Descent for quasi-convex functions with quadratic growth}
            \label{alg2}
            \begin{algorithmic}[1]
              \State \text{Initialize } $x_0 \in \bb R^n$ and $\gamma_0 > 0$. \text{Set } $v_0 = x_0$.
              \If{ $k \ge 0$}
              \State \text{Compute $\alpha_k \in (0, 1)$ satisfying}
              \begin{align*}
          \frac{4L \alpha_k^2}{ \gamma^2} = (1-\alpha_k) \gamma_k +  \alpha_k \mu.
                \end{align*}
                \State Set: $\gamma_{k+1} = (1-\alpha_k)\gamma_k+\alpha_k \mu$.
                \State
                Choose $y_k = v_k + \beta_k(x_k- v_k)$ where
                \begin{align*}
                  \beta_k = \argmin_{\beta \in [0,1]} f(v_k + \beta_k(x_k - v_k)).
                  \end{align*}
                  \State Set: $x_{k+1} = y_k - \frac{1}{L} \nabla f(y_k)$.
                  \State Set:
                  \begin{align*}
                    v_{k+1} = \frac{1}{\gamma_{k+1}} \left[(1-\alpha_k) \gamma_k v_k + \alpha_k \mu y_k - \frac{2\alpha_k}{ \gamma} \nabla f(y_k)\right].
                    \end{align*}
                \EndIf
              \end{algorithmic}
            \end{algorithm}

                          It is straightforward to obtain the following convergence result.
                          \begin{theorem}
                          For the sequence $\{x_k\}$ generated by Algorithm~(\ref{alg2}) one has,
                            \begin{align*}
                              f(x_k) - f(x^*) \le \left( 1 - \frac{1}{2}\sqrt{\frac{\mu \gamma^2}{L}}\right)^k L\|x_k-x_0\|^2.
                              \end{align*}
                            \end{theorem}
                          Guminov {\em et al.}~\cite{guminov2017accelerated} have
                          obtained the same convergence result for this function class via the Conjugate Gradient Method. Their proposed method requires solving a minimization problem over $k$-dimensional subspace at $k^{\text{th}}$ iteration. Moreover, the rate is obtained by applying a restart technique. Here, we only require to perform a line search at each iteration and the  convergence rate is also direct.
                                  \section{Geometric perspective: optimal quadratic averaging for $\ca W_{L, \gamma, \mu}$}
                                  Recently, a number of motivating perspectives on Nesterov's optimal methods have been proposed. In this section, we examine quadratic averaging~\cite{drusvyatskiy2018optimal}, a formulation with clear geometric intuition whose convergence rate matches Nesterov's optimal methods for smooth and strongly convex functions. We shall demonstrate that a slight variation on quadratic averaging would work for the function class $\ca W_{L, \gamma, \mu}$ when $\mu > 0$. Indeed, in~\cite{drusvyatskiy2018optimal}, a line search step is already inherent in the proposed algorithm; one only needs to slightly modify the quadratic underestimate of $\ca W_{L, \gamma, \mu}$. We refer to the paper~\cite{drusvyatskiy2018optimal} for a detailed motivation and only present the essential ideas for $\ca W_{L, \gamma, \mu}$. \par
                                  \subsection{Optimal Quadratic Averaging}
                                  At any point $\bar{x}$, note that the quadratic function given by
                                  \begin{align*}
                                    q(x; \bar{x}) &= f(\bar{x}) + \frac{1}{\gamma} \langle \nabla f(\bar{x}), x-\bar{x} \rangle + \frac{\mu}{2} \|x-\bar{x}\|^2 \\
                                    &= \left( f(\bar{x}) - \frac{1}{2 \mu \gamma^2} \|\nabla f(\bar{x})\|^2 \right) + \frac{\mu}{2} \|x-\bar{x}^{++}\|^2,
                                    \end{align*}
                                      with $\bar{x}^{++} = \bar{x} - \frac{1}{\mu \gamma} \nabla f(\bar{x})$, underestimates $f(x^*)$, i.e., $f(x^*) \ge q(x^*; \bar{x})$. 
                                      The results in~\cite{drusvyatskiy2018optimal} can be smoothly adpated for the function class $\ca W_{L, \gamma, \mu}$. We present the algorithm below.
          \begin{algorithm}[H]
            \caption{Optimal Quadratic Averaging for weakly-quasi-strongly-convex functions}
            \label{alg3}
            \begin{algorithmic}[1]
              \State Initialize $x_0 \in \bb R^n$ and $c_0 = x_0^{++}$.
              \State Set $Q_0(x) = m_0 + \frac{\mu}{2} \|x-c_0\|^2$ where $m_0 = f(x_0) - \frac{\|\nabla f(x_0)\|^2}{2 \mu \gamma^2}$ and $c_0 = x_0^{++}$.
              \If{ $k \ge 1$}
              \State Set: $x_k = {\bf line\_search}(c_{k-1}, x_{k-1}^{+})$.
                \State Set: $Q(x) = \left( f(x_k) - \frac{\|\nabla f(x_k)\|^2}{2 \mu \gamma^2}\right) + \frac{\mu}{2}\|x-x_k^{++}\|^2$.
                \State
                Let $Q_k(x) = m_k + \frac{\alpha}{2} \|x-c_k\|^2$ be the optimal averaging of $Q$ and $Q_{k-1}$.
                \EndIf
              \end{algorithmic}
            \end{algorithm}
                                      \begin{theorem}
                                        In Algorithm~\ref{alg3}, we have
                                        \begin{align*}
                                          f(x_k^+) - m_k \le \left( 1 - \sqrt{\frac{\mu \gamma^2}{L}}\right)^k \left( f(x_0^+) - m_0\right).
                                          \end{align*}
                                        \end{theorem}
                                      The proof in~\cite{drusvyatskiy2018optimal} can be extended to prove the above theorem; we omit the details here.
                                  \section{Concluding remarks and an LQR-related observation}\label{LQR}
                                      As mentioned in the introduction, this work was partially motivated by adopting acceleration for direct policy updates for the linear quadratic regulator (LQR) problem; the reader is referred to ~\cite{bu2019LQR,fazel2018global}.
                                        In order to make this motivation more transparent, we conclude the paper with the following observation.

                                      \begin{lemma}
                                        For the LQR problem, the cost function (as directly parameterized in terms of the feedback gain) satisfies,
                                        \begin{align*}
                                          f(K_*) \ge f(K) + \|Y_*\|_2 \langle \nabla f(K), K-K_*\rangle + \frac{2 \lambda_1(R+B^{\top} X_K B)}{2} \|K-K_*\|_F^2.
                                          \end{align*}
                                          As such, the LQR cost as a function of the feedback gain is $\gamma$-weakly-quasi-convex.
                                        \end{lemma}
                                        
\section*{Acknowledgments}
The authors thank Maryam Fazel for discussions related to the topic of this work, and in particular pointing out the work~\cite{hinder2019}.                                
\bibliographystyle{alpha}
\bibliography{ref}
\end{document}